\documentclass[12pt,english,reqno]{amsart}
\usepackage{a4wide}
\usepackage{amsmath,amssymb,amsfonts,amsthm}
\usepackage{enumerate}
\usepackage{paralist}
\usepackage{mathtools}
\usepackage{graphicx}
\usepackage{caption}

\usepackage[
  pdfauthor={Bernd C. Kellner},
  pdfkeywords={Stirling numbers, Harmonic numbers, Bernoulli numbers, Genocchi numbers},
  pdftitle={Identities between polynomials related to Stirling and harmonic numbers},
  linktocpage,colorlinks,bookmarksnumbered]{hyperref}

\newcommand{\NN}{\mathbb{N}}
\newcommand{\QQ}{\mathbb{Q}}
\newcommand{\RR}{\mathbb{R}}
\newcommand{\ZZ}{\mathbb{Z}}

\newcommand{\ST}{\mathbf{S}}
\newcommand{\HN}{\mathbf{H}}
\newcommand{\BN}{\mathbf{B}}
\newcommand{\BNN}{\widetilde{\mathbf{B}}}
\newcommand{\EN}{\mathbf{E}}
\newcommand{\GN}{\mathbf{G}}
\newcommand{\GNN}{\widetilde{\mathbf{G}}}
\newcommand{\ELN}{\mathbf{A}}
\newcommand{\SF}[2]{\genfrac{\langle}{\rangle}{0pt}{}{#1}{#2}}

\newcommand{\PR}{\mathcal{S}}
\newcommand{\SR}{\mathfrak{S}}
\newcommand{\GF}{\mathcal{G}}
\newcommand{\II}{\mathcal{I}}

\newcommand{\fs}{\mathbf{F}}
\newcommand{\fh}{\widehat{\mathbf{F}}}

\newcommand{\opbr}[1]{\langle #1\rangle}

\DeclareMathOperator{\ord}{ord}

\newtheorem{prop}{Proposition}[section]
\newtheorem{lemma}[prop]{Lemma}
\newtheorem{theorem}[prop]{Theorem}
\newtheorem{corl}[prop]{Corollary}

\theoremstyle{remark}
\newtheorem*{remark*}{Remark}

\theoremstyle{definition}
\newtheorem{tbl}[prop]{Table}

\numberwithin{equation}{section}
\numberwithin{figure}{section}

\title[Identities between polynomials related to Stirling and harmonic numbers]
{Identities between polynomials related to\\Stirling and harmonic numbers}
\author{Bernd C. Kellner}
\email{bk@bernoulli.org}

\begin{document}

\begin{abstract}
We consider two types of polynomials $\mathbf{F}_n (x) = \sum_{\nu=1}^n \nu! \,
\mathbf{S}_2(n,\nu) \, x^\nu$ and $\widehat{\mathbf{F}}_n (x) = \sum_{\nu=1}^n
\nu! \, \mathbf{S}_2(n,\nu) \mathbf{H}_\nu \, x^\nu$, where
$\mathbf{S}_2(n,\nu)$ are the Stirling numbers of the second kind and
$\mathbf{H}_\nu$ are the harmonic numbers. We show some properties and
relations between these polynomials. Especially, the identity
$\widehat{\mathbf{F}}_n (-\tfrac{1}{2}) = - (n-1)/2 \cdot \mathbf{F}_{n-1}
(-\tfrac{1}{2})$ is established for even $n$, where the values are connected
with Genocchi numbers. For odd $n$ the value of $\widehat{\mathbf{F}}_n
(-\tfrac{1}{2})$ is given by a convolution of these numbers. Subsequently, we
discuss some of these convolutions, which are connected with Miki type
convolutions of Bernoulli and Genocchi numbers, and derive some
$2$-adic valuations of them.
\end{abstract}

\maketitle

\section{Introduction}

The purpose of this paper is to show some relations between the polynomials
\[
  \fs_n(x) = \sum_{\nu=1}^n \SF{n}{\nu} x^\nu, \quad
    \fh_n(x) = \sum_{\nu=1}^n \SF{n}{\nu} \HN_\nu x^\nu \quad (n \geq 1).
\]
These polynomials are composed of harmonic numbers
\[
  \HN_n = \sum_{\nu=1}^n \frac{1}{\nu}
\]
and Stirling numbers of the second kind $\ST_2(n,k)$ where we use the related
numbers
\begin{equation} \label{eq:sf-def}
  \SF{n}{k} = k! \, \ST_2(n,k)
\end{equation}
obeying the recurrence
\begin{equation} \label{eq:sf-rec}
  \SF{n+1}{k} = k \left( \SF{n}{k} + \SF{n}{k-1} \right).
\end{equation}
Note that $\ST_2(n,1) = \ST_2(n,n)=1$ for $n \geq 1$ and $\ST_2(n,0)=\SF{n}{0}
= \delta_{n,0}$ for $n \geq 0$ using Kronecker's delta. For properties of
Stirling and harmonic numbers we refer to \cite{Graham&others:1994}. The
polynomials $\fs_n$ are related to the Eulerian polynomials, see
\cite[pp.~243--245]{Comtet:1974} and \cite{Prodinger:1983}, \cite{Tanny:1975}
for a survey. The numbers $\fs_n(1)$ are called ordered Bell numbers or Fubini
numbers, cf. \cite[p.~228]{Comtet:1974}. For a discussion and further
generalizations of the polynomials $\fs_n$ and $\fh_n$ see
\cite{Boyadzhiev:2005} and \cite{Dil&Kurt:2012}, respectively. Note that the
notation $\SF{n}{k}$ is frequently used also for the Eulerian numbers, which we
denote by $\ELN(n,k)$ as in \cite{Comtet:1974}; the notation $\fs_n$ is used as
in \cite{Prodinger:1983}, \cite{Tanny:1975}.

\begin{lemma} \label{lem:fsh-rec}
We have for $n \geq 1$:
\begin{align*}
  \fs_n(-1)    &= (-1)^n, \\
  \fh_n(-1)    &= (-1)^n n, \\
  \fs_{n+1}(x) &= (x^2+x) \fs'_n(x) + x \fs_n(x), \\
  \fh_{n+1}(x) &= (x^2+x) \fh'_n(x) + x \fh_n(x) + x \fs_n(x).
\end{align*}
\end{lemma}

\begin{proof}
The recurrences follow easily by \eqref{eq:sf-rec} and the values at $x=-1$
by induction.
\end{proof}

The Bernoulli numbers $\BN_n$ and the Genocchi numbers $\GN_n$
may be defined by
\begin{align}
  \BN(t) &= \frac{t}{e^t-1} = \sum_{n \geq 0} \BN_n \frac{t^n}{n!}
    \quad (|t| < 2 \pi) \label{eq:gf-bn}
\shortintertext{and}
  \GN(t) &= \frac{2t}{e^t+1} = \sum_{n \geq 0} \GN_n \frac{t^n}{n!}
    \quad (|t| < \pi), \label{eq:gf-gn}
\end{align}
where $\GN_0 = 0$ and $\BN_n = \GN_n = 0$ for odd $n > 1$,
cf.~\cite[pp.~48--49]{Comtet:1974}. Note that
\begin{equation} \label{eq:gn-bn}
  \GN_n = 2 ( 1 - 2^n ) \BN_n \quad (n \geq 0).
\end{equation}
The numbers $\BN_n$ are rational, whereas the numbers $\GN_n$ are integers.
\smallskip

Define the semiring $\PR \subset \ZZ[x]$, which consists of polynomials having
nonnegative integer coefficients. Further define the set
\begin{equation} \label{eq:def-sr}
  \SR_\alpha = \{ f \in \PR : f(\alpha+x) = (-1)^{\deg f} f(\alpha-x)
    \text{ for } x \in \RR \},
\end{equation}
where such polynomials have a reflection relation around $x=\alpha$. For
$f(x)=0$ we declare $\deg f = 0$, such that $0 \in \SR_\alpha$ is well defined.

\begin{theorem} \label{thm:fs-value}
We have the following relations for $n \geq 1$: \smallskip

\begin{compactenum}[(a)]
\item \label{item:fs-1}
      \[
        \int_{-1}^0 \fs_n(x) dx = \BN_n.
      \]
\item \label{item:fs-2}
      \[
        \fs_n (-\tfrac{1}{2}) = \frac{\GN_{n+1}}{n+1}.
      \]
\item \label{item:fs-3}
      \[
        \fs_n(x)/x, \, (x+1) \fs_n(x) \in \SR_{-1/2}.
      \]
\end{compactenum}
\end{theorem}

This theorem can be deduced from known results, which we will give later. By
$\fs_n \in \PR$ and the symmetry property \eqref{item:fs-3} above, we conclude
that $\fs_n(x) > 0$ and $(-1)^n \fs_n(-1-x) > 0$ for $x > 0$; both expressions
strictly increasing as $x \to \infty$. Except for a simple zero at $x=0$, all
real zeros of $\fs_n$ symmetrically lie around $x=-\tfrac{1}{2}$ in the
interval $(-1,0)$. For an illustration see Figure~\ref{fig:fs}. The value of
$\fs_n(-\tfrac{1}{2})$ can be seen as a central value. Note also that $x^2+x
\in \SR_{-1/2}$ occurs in the recurrences of $\fs_n$ and $\fh_n$ given in
Lemma~\ref{lem:fsh-rec}. Interestingly, the integral over the interval $[-1,0]$
and the central value are mainly connected with Bernoulli numbers. Similar
properties also exist for the polynomials $\fh_n$ as follows.

\begin{theorem} \label{thm:fh-value}
We have the following relations: \smallskip

\begin{compactenum}[(a)]
\item \label{item:fh-1}
      \[
        \int_{-1}^0 \fh_n(x) dx = - \frac{n}{2} \BN_{n-1}
        \quad (n \geq 1).
      \]
\item \label{item:fh-2}
      \[
        \fh_n(-\tfrac{1}{2})
          = (-1)^{\delta_{n,1}} \frac{1}{2} \sum_{\nu=1}^n \! \binom{n}{\nu}
          \frac{\GN_\nu}{\nu} \frac{\GN_{n+1-\nu}}{n+1-\nu}
        \quad (\text{odd $n \geq 1$}).
      \]
\item \label{item:fh-3}
      \[
        \fh_n(-\tfrac{1}{2})
          = - \frac{n-1}{2} \, \fs_{n-1}(-\tfrac{1}{2})
        \quad (\text{even $n \geq 2$}).
      \]
\item \label{item:fh-4}
      \[
        \ord_2 \fh_n(-\tfrac{1}{2}) = -1 - \begin{cases}
          \ord_2 n, & \text{if even $n \geq 2$}, \\
          2(r-1),   & \text{if $n = 2^r - 1$ $(r \geq 1)$}, \\
          \ord_2 ( n+1 ) + [ \log_2 ( n+1 ) ],
            & \text{otherwise},
        \end{cases}
      \]
      where $\ord_2$ is the $2$-adic valuation and $[\,\cdot\,]$ gives the
      integer part.
\item \label{item:fh-5}
      \[
        \fh_n(x) = \fs_n(x) + (n-1) x \fs_{n-1}(x)
          + \sum_{\nu = 1}^{n-2} \lambda_{n,\nu}(x) \fs_{\nu}(x)
        \quad (n \geq 2),
      \]
      where $\lambda_{n,\nu} \in \SR_{-1/2}$ and $\deg \lambda_{n,\nu}
      = n - \nu$ for $\nu = 1,\ldots,n-2$.
\item \label{item:fh-6}
      \[
        \fh_n(x)/x - (n-1) \fs_{n-1}(x) \in \SR_{-1/2} \quad (n \geq 2),
      \]
      where the resulting polynomial has degree $n-1$.
\end{compactenum}
\end{theorem}

The polynomials $\lambda_{n,\nu}$ will be recursively defined later in
Proposition~\ref{prop:lambda-rec}. The symmetry of $\fh_n(x)/x - (n-1)
\fs_{n-1}(x)$ is shown in Figure~\ref{fig:fh}. The first relations between the
polynomials $\fh_n$ and $\fs_n$ are given as follows.
\begin{tbl} \label{tbl:fh-fs}
\begin{align*}
  \fh_1(x) &= \fs_1(x), \\
  \fh_2(x) &= \fs_2(x) +  x \fs_1(x), \\
  \fh_3(x) &= \fs_3(x) + 2x \fs_2(x) + (x^2+x) \fs_1(x), \\
  \fh_4(x) &= \fs_4(x) + 3x \fs_3(x) + 3(x^2+x) \fs_2(x) + (2x^3+3x^2+x) \fs_1(x).
\end{align*}
\end{tbl}
\bigskip

In the following theorem a different relation is given by derivatives of $\fs_n$.
\begin{theorem} \label{thm:fsh-deriv}
We have
\[
  \fh_n (x) = \sum_{\nu=1}^n (-1)^{\nu+1} \frac{\fs_n^{(\nu)}(x)}{\nu!}
    \frac{x^\nu}{\nu} \quad (n \geq 1).
\]
\end{theorem}

\begin{remark*}
The identity \eqref{item:fh-3} of Theorem~\ref{thm:fh-value} occurred in
\cite{Kellner:2013} as an important key step in proofs. We shall use different
approaches to prove the theorems in a comprehensive way.
\end{remark*}

\section{Bernoulli and Stirling numbers}

Define
\[
  S_n(m) = \sum_{\nu=0}^{m-1} \nu^n \quad (n \geq 0).
\]
It is well known that
\begin{equation} \label{eq:sum-1}
  S_n(x) = \frac{1}{n+1}( \BN_{n+1}(x) - \BN_{n+1} ),
\end{equation}
where $\BN_n(x)$ is the $n$th Bernoulli polynomial,
cf.~\cite[p.~367]{Graham&others:1994}, with the properties
\begin{equation} \label{eq:bn-poly}
  \BN'_{n+1}(x) = (n+1) \BN_n(x), \quad \BN_n(0) = \BN_n.
\end{equation}
The Gregory-Newton expansion of $x^n$ reads
\begin{equation} \label{eq:sf-binom}
  x^n = \sum_{k=0}^n \SF{n}{k} \binom{x}{k},
\end{equation}
which follows by \eqref{eq:sf-def} and the usual definition of the numbers
$\ST_2(n,k)$ by
\[
  x^n = \sum_{k=0}^n \ST_2(n,k) (x)_k
\]
with falling factorials $(x)_k$. The summation of \eqref{eq:sf-binom}
yields another familiar formula
\begin{equation} \label{eq:sum-2}
  S_n(x) = \sum_{k=0}^{n} \SF{n}{k} \binom{x}{k+1}.
\end{equation}
Note that $\SF{n}{0} = 0$ for $n \geq 1$ and $\binom{-1}{k} = (-1)^k$.
The following formula is a classical result which is due to Worpitzky.
We give a short proof.

\begin{prop}[Worpitzky {\cite[(36), p.~215]{Worpitzky:1883}}] \label{prop:drv-fs-bn}
We have
\[
  \sum_{k=1}^n \SF{n}{k} \frac{(-1)^k}{k+1} = \BN_n \quad (n \geq 1).
\]
\end{prop}

\begin{proof}
Since $S_n(0) = 0$, we conclude by \eqref{eq:sum-1}, \eqref{eq:bn-poly}, and
\eqref{eq:sum-2} that
\[
  \BN_n = S_n'(0) = \lim_{x \to 0} S_n(x)/x
    = \lim_{x \to 0} \sum_{k=0}^{n} \SF{n}{k} \frac{1}{k+1} \binom{x-1}{k}
    = \sum_{k=1}^{n} \SF{n}{k} \frac{(-1)^k}{k+1}. \qedhere
\]
\end{proof}

A similar result with harmonic numbers is the following.

\begin{prop} \label{prop:drv-fh-bn}
We have
\[
  \sum_{k=1}^n \SF{n}{k} \HN_k \frac{(-1)^k}{k+1}
    = - \frac{n}{2} \BN_{n-1} \quad (n \geq 1).
\]
\end{prop}

\begin{proof}
The derivative of \eqref{eq:sum-2} provides that
\begin{align*}
  S'_n(x) &= \sum_{k=0}^{n} \SF{n}{k} \binom{x}{k+1}
    \sum_{j=0}^k \frac{1}{x-j} = S_n(x)/x - x V_n(x)
\shortintertext{where}
  V_n(x)  &= \sum_{k=0}^{n} \SF{n}{k} \frac{1}{k+1}
    \binom{x-1}{k} \sum_{j=1}^k \frac{1}{j-x}.
\end{align*}
Since $|V_n(0)| < \infty$ and $S_n(x) - x S'_n(x) \to 0$ as $x \to 0$,
we obtain by L'H\^{o}pital's rule that
\[
  V_n(0) = \lim_{x \to 0} \frac{S_n(x) - x S'_n(x)}{x^2}
    = \lim_{x \to 0} \frac{- x S''_n(x)}{2x} = - \frac{1}{2} S''_n(0).
\]
Using \eqref{eq:sum-1} and \eqref{eq:bn-poly} we then derive that
\[
  S''_n(x) = n \BN_{n-1}(x) \quad \text{and} \quad
    V_n(0) = - \frac{n}{2} \BN_{n-1}.
\]
This shows the claimed identity.
\end{proof}

\section{Symmetry properties}
\label{sect:symm}

We shall give some symmetry relations of the polynomials $\fs_n$ and $\fh_n$.
Note that the Eulerian numbers, as used in \cite{Comtet:1974},
\cite{Tanny:1975}, are symmetric such that $\ELN(n,k) = \ELN(n,n-k)$. Recall
the definition of $\SR_\alpha$ in \eqref{eq:def-sr}.

\begin{lemma} \label{lem:semiring}
The set $\SR_\alpha$ has pseudo semiring properties.
If $f, g \in \SR_\alpha$, then
\begin{align*}
  f \cdot g &\in \SR_\alpha, \\
  f + g     &\in \SR_\alpha, \quad (*) \\
  f'        &\in \SR_\alpha,
\end{align*}
where in case of addition and $f \cdot g \neq 0$ a parity condition
must hold such that
\[
  (*) \quad \deg f \equiv \deg g \pmod{2}.
\]
If $f$ has odd degree, then $f(\alpha)=0$.
\end{lemma}

\begin{proof}
The cases, where $f=0$ or $g=0$, are trivial. Since $\SR_\alpha \subset \PR$,
the pseudo semiring properties follow by the parity of $(-1)^{\deg f}$,
resp., $(-1)^{\deg g}$. If $\deg f$ is odd, then $f(\alpha) = - f(\alpha)$,
which implies that $f(\alpha)=0$.
\end{proof}

\begin{prop}[Tanny {\cite[(16), p.~737]{Tanny:1975}}]
We have
\[
  \fs_n (x) = \sum_{k=1}^n \ELN(n,k) x^{n-k+1} (x+1)^{k-1} \quad (n \geq 1).
\]
\end{prop}

\begin{corl} \label{corl:fs-sr}
We have
\[
  \fs_n(x)/x, \, (x+1) \fs_n(x) \in \SR_{-1/2} \quad (n \geq 1).
\]
\end{corl}

\begin{proof}
Using the symmetry of $\ELN(n,k)$, we obtain that
\[
   f_n(x) := \fs_n(x)/x = \sum_{k=1}^n \ELN(n,k) x^{n-k} (x+1)^{k-1}
     = (-1)^{n-1} f_n(-(x+1)).
\]
Hence, $f_n(-\tfrac{1}{2}+x) = (-1)^{n-1} f_n(-\tfrac{1}{2}-x)$ and $\deg f_n =
n-1$ show that $\fs_n(x)/x \in \SR_{-1/2}$. Since $x^2+x \in \SR_{-1/2}$,
it also follows that
\[
  (x^2+x) \cdot \fs_n(x)/x = (x+1) \fs_n(x) \in \SR_{-1/2}. \qedhere
\]
\end{proof}

\begin{prop} \label{prop:lambda-rec}
We have for $n \geq 1$:
\[
  \fh_n (x) = \sum_{\nu=1}^n \lambda_{n,\nu}(x) \fs_\nu( x )
\]
where
\begin{align*}
  \lambda_{n,\nu}(x) &= \begin{cases}
       1, & \text{if $n=\nu=1$}, \\
       0, & \text{if $\nu \notin \{1,\ldots,n\}$}, \\
    \end{cases}
\intertext{otherwise recursively defined by}
  \lambda_{n+1,\nu}(x) &= (x^2+x) \lambda_{n,\nu}'(x) + \lambda_{n,\nu-1}(x)
    + \delta_{n,\nu} x.
\end{align*}
Furthermore $\lambda_{n,\nu} \in \PR$ and $\deg \lambda_{n,\nu} = n - \nu$
for $\nu = 1,\ldots,n$. Especially
\[
  \lambda_{n,n-1}(x) = (n-1)x.
\]
\end{prop}

\begin{proof}
We use induction on $n$. For $n = 1$ we have
\[
  \fh_1 (x) = \fs_1( x ) \quad \text{and} \quad
    \lambda_{1,\nu}(x) = \delta_{1,\nu}.
\]
Now assume the result is true for $n$. By assumption we have
\begin{align*}
  \fh_n (x)    &= \sum_{\nu=1}^n \lambda_{n,\nu}(x) \fs_\nu( x ), \\
  \fh'_n (x)   &= \sum_{\nu=1}^n \lambda'_{n,\nu}(x) \fs_\nu( x )
    + \lambda_{n,\nu}(x) \fs'_\nu( x ),
\shortintertext{and by Lemma~\ref{lem:fsh-rec} that}
  \fh_{n+1}(x) &= (x^2+x) \fh'_n(x) + x \fh_n(x) + x \fs_n(x), \\
    (x^2+x) \fs'_n(x) &= \fs_{n+1}(x) - x \fs_n(x).
\end{align*}
It follows that \allowdisplaybreaks
\begin{align*}
  \fh_{n+1}(x)
    &= (x^2+x) \sum_{\nu=1}^n \left( \lambda'_{n,\nu}(x) \fs_\nu( x )
      + \lambda_{n,\nu}(x) \fs'_\nu( x ) \right) + x \fh_n(x) + x \fs_n(x) \\
    &= (x^2+x) \sum_{\nu=1}^n \lambda'_{n,\nu}(x) \fs_\nu( x )
       + \sum_{\nu=1}^n \lambda_{n,\nu}(x) \left( \fs_{\nu+1}(x) - x \fs_\nu(x)
       \right) \\
    &\quad + x \fh_n(x) + x \fs_n(x) \\
    &= (x^2+x) \sum_{\nu=1}^n \lambda'_{n,\nu}(x) \fs_\nu( x ) + \sum_{\nu=1}^n
       \lambda_{n,\nu}(x) \fs_{\nu+1}(x)  + x \fs_n(x) \\
    &= \sum_{\nu=1}^{n+1} \lambda_{n+1,\nu}(x) \fs_\nu( x ).
\end{align*}
Thus
\begin{align}
  \lambda_{n+1,\nu}(x) &= (x^2+x) \lambda_{n,\nu}'(x) + \lambda_{n,\nu-1}(x)
    + \delta_{n,\nu} x. \nonumber
\shortintertext{In particular, we have}
  \lambda_{n+1,n+1}(x) &= \lambda_{n,n}(x) = 1
  \label{eq:loc-lambda-1}
\shortintertext{and}
  \lambda_{n+1,n}(x)   &= \lambda_{n,n-1}(x) + x = (n-1)x + x = nx.
  \label{eq:loc-lambda-2}
\end{align}
The recurrence shows that $\lambda_{n+1,\nu} \in \PR$ for $\nu = 1,\ldots,n+1$.
Therefore we conclude for $1 \leq \nu < n$ that
\[
  \deg \lambda_{n+1,\nu}
    = \max( 2 + \deg \lambda_{n,\nu}', \lambda_{n,\nu-1} ) = n-\nu + 1.
\]
Along with \eqref{eq:loc-lambda-1} and \eqref{eq:loc-lambda-2} this shows
the claimed properties for $n+1$.
\end{proof}

\begin{prop} \label{prop:lambda-sr}
We have $\lambda_{n,\nu} \in \SR_{-1/2}$ for $n-2 \geq \nu \geq 1$.
\end{prop}

\begin{proof}
We make use of Proposition~\ref{prop:lambda-rec} and Lemma~\ref{lem:semiring}.
For $n-1 \geq \nu \geq 1$ we have
\[
  \lambda_{n+1,\nu}(x) = (x^2+x) \lambda_{n,\nu}'(x) + \lambda_{n,\nu-1}(x).
\]
We use induction on $n$. For $n = 3$ we have
\[
  \lambda_{3,1}(x) = (x^2+x) \lambda_{2,1}'(x) + \lambda_{2,0}(x)
    = x^2+x \in \SR_{-1/2}.
\]
Now assume the result holds for $n \geq 3$. For $\nu = 1, \ldots, n-2$ we have
\[
  \lambda_{n+1,\nu}(x) = (x^2+x) \lambda_{n,\nu}'(x) + \lambda_{n,\nu-1}(x)
    \in \SR_{-1/2},
\]
since $\lambda_{n,\nu}', \lambda_{n,\nu-1} \in \SR_{-1/2}$ by assumption
and $2 + \deg \lambda_{n,\nu}' = \deg \lambda_{n,\nu-1}$ if $\nu \neq 1$,
otherwise $\lambda_{n,\nu-1} = 0$. It remains the case $\nu = n-1$:
\[
  \lambda_{n+1,n-1}(x) = (x^2+x) (n-1) + \lambda_{n,n-2}(x) \in \SR_{-1/2},
\]
since $\lambda_{n,n-1}'(x) = n-1$ and $\lambda_{n,n-2} \in \SR_{-1/2}$
with $\deg \lambda_{n,n-2} = 2$.
\end{proof}

\begin{corl} \label{corl:fh-sr}
We have
\[
  \fh_n(x)/x - (n-1) \fs_{n-1}(x) \in \SR_{-1/2} \quad (n \geq 2),
\]
where the resulting polynomial has degree $n-1$.
\end{corl}

\begin{proof}
Propositions \ref{prop:lambda-rec} and \ref{prop:lambda-sr} show that
\begin{equation} \label{eq:loc-fh-sr}
  \fh_n(x)/x - (n-1) \fs_{n-1}(x) = \fs_n(x)/x
    + \sum_{\nu = 1}^{n-2} \lambda_{n,\nu}(x) \cdot \fs_{\nu}(x)/x,
\end{equation}
where $\lambda_{n,\nu} \in \SR_{-1/2}$ and $\deg \lambda_{n,\nu} = n - \nu$ for
$\nu = 1,\ldots,n-2$. By Corollary~\ref{corl:fs-sr} and
Lemma~\ref{lem:semiring} we conclude that $\lambda_{n,\nu}(x) \cdot
\fs_{\nu}(x)/x \in \SR_{-1/2}$. Since the first term and the products on the
right-hand side of \eqref{eq:loc-fh-sr} lie in $\SR_{-1/2}$ having the same
degree $n-1$, the left-hand side of \eqref{eq:loc-fh-sr} also lies in
$\SR_{-1/2}$ and has degree $n-1$.
\end{proof}

While $\fs_n(x)/x \in \SR_{-1/2}$ for $n \geq 1$, we have $\fh_n(x)/x \notin
\SR_{-1/2}$ for $n \geq 2$; compare values at $x=0$ and $x=-1$ using
Lemma~\ref{lem:fsh-rec}. The latter function needs a correction term to lie in
$\SR_{-1/2}$. For an illustration of the symmetry of $\fs_n(x)/x$ and
$\fh_n(x)/x - (n-1) \fs_{n-1}(x)$ see Figures \ref{fig:fs} and \ref{fig:fh}.

\section{Generating functions}

Let $[t^n]$ be the linear operator, that gives the coefficient of $t^n$
of a formal power series, such that
\[
  f(t) = \sum_{n \geq 0} a_n t^n, \quad [t^n] \, f(t) = a_n,
\]
see \cite[p.~197]{Graham&others:1994}. Define $\opbr{t^n} = n! [t^n]$. Let
$(h_n(x))_{n \geq 0}$ be a sequence of functions. Then we denote by
\[
  \GF h(x,t) = \sum_{n \geq 0} h_n(x) \frac{t^n}{n!}
\]
the two-variable exponential generating function, such that
\[
  \opbr{t^n} \, \GF h(x,t) = h_n(x).
\]

Recall the definitions of $\BN(t)$ and $\GN(t)$ in \eqref{eq:gf-bn} and
\eqref{eq:gf-gn}, respectively. We further set
\[
  \BNN(t) = \BN(t)/t, \quad \GNN(t) = \GN(t)/t,
\]
and
\[
  \BNN_0 = \GNN_0 = 0, \quad \BNN_n = \BN_n/n, \quad
    \GNN_n = \GN_n/n \quad (n \geq 1).
\]

\begin{lemma} \label{lem:gnn-diff-eq}
The function $y(t)=\GNN(t)$ satisfies the Bernoulli differential equation
\[
  y' + y = \frac{1}{2} y^2, \quad \text{which is equivalent to} \quad
    (\log y)' = \frac{1}{2} y - 1.
\]
\end{lemma}

\begin{proof}
Both equations are identical due to $(\log y)' = y'/y$
and are verified by $y(t) = 2/(e^t+1)$.
\end{proof}

\begin{corl} \label{corl:log-gnn}
We have
\begin{alignat}{3}
  \GNN(t) &= \sum_{n \geq 0} \GNN_{n+1} \frac{t^n}{n!}
    \quad && (|t| < \pi), \label{eq:gf-gnn} \\
  \log \GNN(t) &= \frac{1}{2} \sum_{n \geq 1} (-1)^n \GNN_n \frac{t^n}{n!}
    \quad && (|t| < \pi). \label{eq:gf-gnn-log}
\end{alignat}
\end{corl}

\begin{proof}
Eq.~\eqref{eq:gf-gnn} follows by its definition. Integrating the right-hand
side differential equation of Lemma~\ref{lem:gnn-diff-eq}, we derive that
\[
  \log \GNN(t) = \int \left( \frac{1}{2}\GNN(t) - 1 \right) dt
    = \frac{1}{2} \sum_{n \geq 1} \GNN_n \frac{t^n}{n!} - t + C
\]
with a constant $C$. Since $\log \GNN(0) = 0$, we obtain $C=0$.
Note that $(-1)^n \GNN_n = \GNN_n$ for $n \geq 2$ and $\GNN_1 = 1$.
With $t \, \GNN_1/2 - t = -t \, \GNN_1/2$ we finally get \eqref{eq:gf-gnn-log}.
\end{proof}

We also have a connection with hyperbolic functions, where we casually obtain
the known coefficients of the following function by \eqref{eq:gf-gnn-log}.
\begin{lemma} \label{lem:gnn-hyper}
We have
\[
  \log \GNN(t) = - \frac{t}{2} - \log \cosh \left( \frac{t}{2} \right).
\]
\end{lemma}

\begin{proof}
This follows by
\[
  e^{t/2} \cosh \left( \frac{t}{2} \right) = \frac{e^t+1}{2} = \GNN(t)^{-1}.
    \qedhere
\]
\end{proof}

\begin{prop} \label{prop:gf-fn}
Define $\fs_0(x)=\fh_0(x)=1$ and $\psi(x,t) = 1-x(e^t-1)$. Then we have
\begin{enumerate}[(a)]
\item \label{item:gf-f-1}
      \[
        \GF \fs(x,t) = \frac{1}{\psi(x,t)},
          \quad \int \GF \fs(x,t) dx = - \BNN(t) \log \psi(x,t),
      \]
\item \label{item:gf-f-2}
      \[
        \GF \fh(x,t) = - \frac{\log \psi(x,t)}{\psi(x,t)},
          \quad \int \GF \fh(x,t) dx = \frac{1}{2} \BNN(t)
          \left( \log \psi(x,t) \right)^2.
      \]
\end{enumerate}
\end{prop}

\begin{proof}
Set $u = e^t-1$. Note that $1/u = \BNN(t)$ and $1 - xu = \psi(x,t)$. We need
the following generating functions (cf. \cite[p.~351]{Graham&others:1994}):
\begin{align}
  (e^t-1)^k &= \sum_{n \geq k} \SF{n}{k} \frac{t^n}{n!}, \label{eq:gf-sf} \\
  - \frac{\log(1-t)}{1-t} &= \sum_{n \geq 1} \HN_n t^n. \label{eq:gf-hn}
\end{align}
\eqref{item:gf-f-1} Using \eqref{eq:gf-sf} we obtain that
\begin{equation} \label{eq:loc-gf-fs}
  \frac{1}{1-x(e^t-1)} = \sum_{k \geq 0} (x(e^t-1))^k
    = \sum_{k \geq 0} x^k \sum_{n \geq k} \SF{n}{k} \frac{t^n}{n!}
    = \sum_{n \geq 0} \fs_n(x) \frac{t^n}{n!} = \GF \fs(x,t).
\end{equation}
We further deduce that
\[
  \int \GF \fs(x,t) dx = \int \frac{dx}{1 - xu} = - \frac{\log(1 - xu)}{u}
    = - \BNN(t) \log \psi(x,t).
\]
\eqref{item:gf-f-2} First substitute $t$ by $x(e^t-1)$ in \eqref{eq:gf-hn}.
The result for $\GF \fh(x,t)$ is similarly derived as in \eqref{eq:loc-gf-fs}
with an additional factor $\HN_k$. The integral follows by
\[
  \int \frac{\log(1 - xu)}{1 - xu} dx = - \frac{\log(1 - xu)^2}{2u}. \qedhere
\]
\end{proof}

\begin{prop} \label{prop:val-fn}
We have
\begin{align*}
  \fs_n(-\tfrac{1}{2}) &= \GNN_{n+1} \quad (n \geq 1), \\
  \fh_n(-\tfrac{1}{2}) &= \begin{cases}
    -\frac{1}{2}, & \text{if $n = 1$}, \\
    -\frac{n-1}{2} \GNN_n, & \text{if even $n \geq 2$}, \\
    \frac{1}{2} \sum\limits_{\nu=1}^n \! \binom{n}{\nu}
    \GNN_\nu \GNN_{n+1-\nu}, & \text{if odd $n \geq 3$}.
  \end{cases}
\end{align*}
\end{prop}

\begin{proof}
From Proposition~\ref{prop:gf-fn} we conclude that
\[
  \psi(-\tfrac{1}{2},t) = \frac{e^t+1}{2} = \GNN(t)^{-1}.
\]
Hence,
\[
  \GF \fs(-\tfrac{1}{2},t) = \GNN(t) \quad \text{and} \quad
    \GF \fh(-\tfrac{1}{2},t) = \GNN(t) \log \GNN(t).
\]
By \eqref{eq:gf-gnn} we derive that
\[
  \fs_n(-\tfrac{1}{2}) = \opbr{t^n} \, \GF \fs(-\tfrac{1}{2},t)
    = \opbr{t^n} \, \GNN(t) = \GNN_{n+1}.
\]
Similarly, we obtain that
\[
  \fh_n(-\tfrac{1}{2}) = \opbr{t^n} \, \GF \fh(-\tfrac{1}{2},t)
    = \opbr{t^n} \, \GNN(t) \log \GNN(t) = c_n,
\]
where $c_n$ arises from the convolution sum caused by the Cauchy product of $\GNN(t)$
and $\log \GNN(t)$. By means of \eqref{eq:gf-gnn} and \eqref{eq:gf-gnn-log} we
achieve that
\[
  c_n = \frac{1}{2} \sum_{\nu=1}^n \binom{n}{\nu}
    (-1)^\nu \GNN_\nu \GNN_{n+1-\nu}.
\]
Case even $n$: The indices $\nu$ and $n+1-\nu$ have different parity.
Since $\GNN_\nu = 0$ for odd $\nu > 1$ and $\GNN_1 = 1$, the sum simplifies to
\[
  c_n = \frac{1}{2} \left( - \binom{n}{1} \GNN_1 \GNN_n
    + \binom{n}{n} \GNN_n \GNN_1 \right)
    = - \frac{n-1}{2} \GNN_n.
\]
Case odd $n$: For $n=1$ we compute $c_1 = -\tfrac{1}{2}$. Let $n \geq 3$.
Because of the same parity of the indices, we finally infer that
\[
  c_n = \frac{1}{2} \sum_{\nu=1}^n \binom{n}{\nu}  \GNN_\nu \GNN_{n+1-\nu}
\]
by omitting the factor $(-1)^\nu$.
\end{proof}

\begin{prop} \label{prop:int-fn}
We have
\[
  \int_{-1}^0 \fs_n(x) dx = \BN_n \quad \text{and} \quad
    \int_{-1}^0 \fh_n(x) dx = -\frac{n}{2} \BN_{n-1} \quad (n \geq 1).
\]
\end{prop}

\begin{proof}
Using Proposition~\ref{prop:gf-fn}, we obtain that $\psi(0,t) = 1$ and
$\psi(-1,t) = e^t$. Therefore
\begin{align*}
  \int_{-1}^0 \GF \fs(x,t) dx &= - \BNN(t) \log \psi(x,t) \Big|_{x=-1}^{0}
    = \BNN(t) t = \BN(t)
\shortintertext{and}
  \int_{-1}^0 \GF \fh(x,t) dx &= \frac{1}{2} \BNN(t)
    \left( \log \psi(x,t) \right)^2 \Big|_{x=-1}^{0}
    = - \frac{1}{2} \BNN(t) t^2 = - \frac{t}{2} \BN(t).
\end{align*}
Since the integrals above are independent of $t$,  the operator $\opbr{t^n}$
commutes with integration. Applying $\opbr{t^n}$ to these equations easily
yields the results.
\end{proof}

\begin{remark*}
The generating function $\GF \fs(x,t)$ can be found in \cite[(9), p.~736]{Tanny:1975}
and \cite[(3.14), p.~3853]{Boyadzhiev:2005}. The value of $\fs_n(-\tfrac{1}{2})$ was
posed as an exercise in \cite[6.76, p.~559]{Graham&others:1994} and also given
in \cite[(3.29), p.~3855]{Boyadzhiev:2005}; see \cite[p.~288]{Sprugnoli:1994} for a
short proof using the theory of Riordan arrays.
\end{remark*}

\section{Convolutions}

We use the notations
\[
  \begin{aligned}
    ( \alpha_r + \beta_s )^n
      &= \sum_{\nu=0}^n \binom{n}{\nu} \alpha_{r+\nu} \beta_{s+n-\nu}, \\
    \{ \alpha_r + \beta_s \}^n
      &= \sum_{\nu=0}^n \alpha_{r+\nu} \beta_{s+n-\nu}
  \end{aligned}
  \quad (n, r, s \geq 0)
\]
for symmetric binomial, resp.\ usual convolutions of two sequences
$(\alpha_\nu)_{\nu \geq 0}$ and $(\beta_\nu)_{\nu \geq 0}$. Unless otherwise
noted, we generally assume $n \geq 1$ for convolutions in this section. We
first need a simple lemma (cf.~\cite[p.~82]{Gessel:2005}).

\begin{lemma} \label{lem:conv-rel}
If $n \geq 1$ and $r,s \geq 0$, then
\begin{align*}
  (\alpha_r + \beta_s)^n &= (\alpha_r + \beta_{s+1})^{n-1}
    + (\alpha_{r+1} + \beta_s)^{n-1}, \\
  (\widetilde{\alpha}_0 + \widetilde{\beta}_0)^n
    &= \frac{1}{n} (\widetilde{\alpha}_0 + \beta_0)^n
    + \frac{1}{n} (\alpha_0 + \widetilde{\beta}_0)^n,
\end{align*}
where for the second part $\alpha_0 = \widetilde{\alpha}_0 = \beta_0 =
\widetilde{\beta}_0 = 0$ and $\widetilde{\alpha}_\nu = \alpha_\nu / \nu$,
$\widetilde{\beta}_\nu = \beta_\nu / \nu$ for $\nu \geq 1$.
\end{lemma}

\begin{proof}
The first part follows by $\binom{n}{\nu} = \binom{n-1}{\nu} + \binom{n-1}{\nu-1}$,
the second part by the identity $1/(\nu (n-\nu)) = 1 / (n \nu) + 1 / (n (n-\nu))$.
\end{proof}

The Euler polynomials $\EN_n(x)$ are defined by
\begin{equation} \label{eq:gf-en}
  \frac{2e^{xt}}{e^t+1} = \sum_{n \geq 0} \EN_n(x) \frac{t^n}{n!}
    \quad (|t| < \pi).
\end{equation}
Compared to \eqref{eq:gf-gn} and \eqref{eq:gf-gnn},
formulas with $\EN_n(x)$ are naturally transferred to
Genocchi numbers by the relation $\EN_n(0) = \GNN_{n+1}$.
The well-known identity (\cite[(17), p.~135]{Norlund:1920},
\cite[(51.6.2), p.~346]{Hansen:1975})
\[
  (\EN_0(x) + \EN_0(y))^n = 2(1-x-y)\EN_n(x+y) + 2\EN_{n+1}(x+y)
\]
leads to basic convolutions of the Genocchi numbers
\begin{align}
  (\GNN_1 + \GNN_1)^n &= 2 \GNN_{n+2} + 2 \GNN_{n+1}, \label{eq:conv-gnn-11} \\
  (\GNN_1 + \GNN_2)^n &= \GNN_{n+3} + \GNN_{n+2}, \label{eq:conv-gnn-12}
\end{align}
where the latter equation is derived by Lemma~\ref{lem:conv-rel}. Note that
\eqref{eq:conv-gnn-11} also follows immediately by Lemma~\ref{lem:gnn-diff-eq}
and \eqref{eq:gf-gnn}. More general convolution identities can be found in
\cite{Chu&Zhou:2012} for Bernoulli and Euler polynomials, that cover some known
convolutions as special cases.
\smallskip

As a result of Proposition \ref{prop:val-fn} in the last section, the convolution
\begin{equation} \label{eq:conv-gnn-0}
  ( \GNN_0 + \GNN_1 )^n = \frac{1}{2} ( \GNN_0 + \GNN_0 )^{n+1}
\end{equation}
has appeared, which resists a simple evaluation for odd $n > 1$; the right-hand
side follows by Lemma~\ref{lem:conv-rel}. We shall give some arguments in the
following that this remains an open problem. Convolutions with different
indices in the shape of $( \GN_j + \GN_k )^n$ for $j+k \geq 0$ and $( \GNN_1 +
\GNN_k )^n$ for $k \geq 1$ are discussed in \cite{Agoh&Dilcher:2007} and
\cite{Agoh:1989}, respectively. A connection with the last-mentioned
convolutions is established by the following lemma.

\begin{lemma} \label{lem:conv-sum-gnn}
If $n \geq 1$, then
\[
  ( \GNN_0 + \GNN_1 )^n = \sum_{k=1}^{n} ( \GNN_1 + \GNN_k )^{n-k}.
\]
\end{lemma}

\begin{proof}
By Lemma~\ref{lem:conv-rel} we obtain the recurrences
\[
  ( \GNN_0 + \GNN_k )^{n-k+1} = ( \GNN_0 + \GNN_{k+1} )^{n-k}
    + ( \GNN_1 + \GNN_k )^{n-k}
\]
for $k=1,\ldots,n$. Since $( \GNN_0 + \GNN_{n+1} )^0 = 0$,
this gives the claimed sum.
\end{proof}

However, this will not simplify \eqref{eq:conv-gnn-0}, see below. As before, we
set $\BNN_0(x) = 0$ and $\BNN_n(x) = \BN_n(x)/n$ $(n \geq 1)$ for the Bernoulli
polynomials. We can translate \eqref{eq:conv-gnn-0} as follows.

\begin{lemma} \label{lem:conv-gnn-00}
We have
\[
  \frac{1}{4} ( \GNN_0 + \GNN_0 )^n = ( \BNN_0 + \BNN_0 )^n
    - 2^n ( \BNN_0 + \BNN_0(\tfrac{1}{2}) )^n
    \quad (n \geq 1).
\]
\end{lemma}

\begin{proof}
By \eqref{eq:gn-bn} we have $\GNN_\nu \GNN_{n-\nu} / 4 = ( 1 - 2^\nu -
2^{n-\nu} + 2^n ) \BNN_\nu \BNN_{n-\nu}$. It is well known that
$\BN_n(\tfrac{1}{2}) = (2^{1-n}-1) \BN_n$. Observing the symmetry, the result
follows from $( 2^{n-\nu} - 2^{n-1} ) \BNN_\nu = 2^{n-1} ( 2^{1-\nu}-1 )
\BNN_\nu = 2^{n-1} \BNN_\nu(\tfrac{1}{2})$.
\end{proof}

\begin{prop}[Gessel {\cite[(12), p.~81]{Gessel:2005}}]
If $n \geq 1$, then
\[
  \frac{n}{2} \{ \BNN_0(x) + \BNN_0(x) \}^n - ( \BNN_0 + \BN_0(x) )^n
    = \HN_{n-1} \BN_n(x) + \frac{n}{2} \BN_{n-1}(x).
\]
For $x=0$ this reduces to Miki's identity \cite[Theorem, p.~297]{Miki:1978}:
\[
  \{ \BNN_0 + \BNN_0 \}^n - ( \BNN_0 + \BNN_0 )^n
    = 2 \HN_n \BNN_n + \BN_{n-1}.
\]
For $x=\tfrac{1}{2}$ this gives the Faber–-Pandharipande–-Zagier identity
\cite[Lemma~4, p.~22]{Faber&Pandharipande:2000}:
\[
  \frac{n}{2} \{ \BNN_0(\tfrac{1}{2}) + \BNN_0(\tfrac{1}{2}) \}^n
    - ( \BNN_0 + \BN_0(\tfrac{1}{2}) )^n
    = \HN_{n-1} \BN_n(\tfrac{1}{2}) + \frac{n}{2} \BN_{n-1}(\tfrac{1}{2}).
\]
\end{prop}

Note that $( \GNN_0 + \GNN_0 )^n$ is mainly transferred to $( \BNN_0 + \BNN_0 )^n$,
$( \BNN_0 + \BN_0(\tfrac{1}{2}) )^n$, and $( \BN_0 + \BNN_0(\tfrac{1}{2}) )^n$
by Lemmas~\ref{lem:conv-gnn-00} and \ref{lem:conv-rel}.
The Miki type convolutions above show that one can replace binomial convolutions
by usual convolutions, but this does not lead to a simplification compared to
\eqref{eq:conv-gnn-11}.
\smallskip

Agoh showed the formula below for $k \geq 3$ in a slightly different form. In
view of \eqref{eq:conv-gnn-11} and \eqref{eq:conv-gnn-12} this is also valid
for $k=1, 2$.

\begin{prop}[{Agoh \cite[Theorem, p.~61]{Agoh:1989}}] \label{prop:conv-gnn-1k}
If $k, n \geq 1$, then
\[
  ( \GNN_1 + \GNN_k )^n
    = 2 \left( \GNN_{n+k} - \BNN_k \GNN_{n+1}
      + \frac{1}{k} ( \BN_0 + \GNN_{n+1} )^k \right).
\]
\end{prop}

As an application we derive a different Miki type convolution as follows.
\begin{prop}
For $n \geq 1$ we have
\[
  ( \GNN_0 + \GNN_1 )^n
    = 2 \left( \HN_{n-1} \GNN_{n+1} -\frac{1}{2} \GNN_n + \GN_n
      - \{ \BNN_0 + \GNN_2 \}^{n-1}
      + ( \BNN_0 + \GNN_2 )^{n-1} \right).
\]
\end{prop}

\begin{proof}
The case $n=1$ is trivial, let $n \geq 2$. We combine
Lemma~\ref{lem:conv-sum-gnn} and Proposition~\ref{prop:conv-gnn-1k}, where we
split the summations. Note that $\BNN_0=0$ and $\BN_0 = 1$. We then obtain that
\[
  ( \GNN_0 + \GNN_1 )^n = \GNN_n + \sum_{k=1}^{n-1} ( \GNN_1 + \GNN_{n-k} )^k
    = \GNN_n + 2 ( S_1 + S_2 + S_3 ),
\]
where
\begin{align*}
  S_1 &= (n-1) \GNN_n, \quad
    S_2 = - \sum_{k=1}^{n-1} \BNN_k \GNN_{n+1-k}
        = - \{ \BNN_0 + \GNN_2 \}^{n-1}, \\
  S_3 &= \sum_{k=1}^{n-1} \frac{1}{k}
         \sum_{\nu=0}^k \binom{k}{\nu} \BN_\nu \GNN_{n+1-\nu}
      = \HN_{n-1} \GNN_{n+1} + \sum_{k=1}^{n-1} \sum_{\nu=1}^k
         \binom{k-1}{\nu-1} \BNN_\nu \GNN_{n+1-\nu} \\
      &= \HN_{n-1} \GNN_{n+1} + \sum_{\nu=0}^{n-1} \binom{n-1}{\nu}
         \BNN_\nu \GNN_{n+1-\nu}
      = \HN_{n-1} \GNN_{n+1} + ( \BNN_0 + \GNN_2 )^{n-1}.
\end{align*}
Summing up the terms establishes the result.
\end{proof}

As a result of Section~\ref{sect:symm}, we have yet another formula.
\begin{prop}
For odd $n > 1$ we have
\[
  ( \GNN_0 + \GNN_1 )^n = 2 \left( \GNN_{n+1} +
    \sum_{\substack{\nu=1\\\text{odd $\nu$}}}^{n-2}
    \lambda_{n,\nu}(-\tfrac{1}{2}) \GNN_{\nu+1}
    \right),
\]
where the polynomials $\lambda_{n,\nu} \in \SR_{-1/2}$ are defined as in
Proposition~\ref{prop:lambda-rec}.
\end{prop}

\begin{proof}
By Propositions \ref{prop:lambda-rec} and \ref{prop:lambda-sr} we have
\[
  \fh_n (x) = \fs_n( x ) + (n-1) x \fs_{n-1}(x) +
    \sum_{\nu=1}^{n-2} \lambda_{n,\nu}(x) \fs_\nu( x ),
\]
where  $\lambda_{n,\nu} \in \SR_{-1/2}$. Along with
Proposition~\ref{prop:val-fn} we derive the result by setting $x=-\tfrac{1}{2}$
and omitting the terms where $\GNN_{\nu+1} = 0$.
\end{proof}

The last two propositions show that one can resolve the convolution in
\eqref{eq:conv-gnn-0} by
\[
  ( \GNN_0 + \GNN_1 )^n = \gamma_{n+1} \GNN_{n+1} + \gamma_{n-1} \GNN_{n-1} +
    \cdots + \gamma_{2} \GNN_{2} \quad (\text{odd $n > 1$})
\]
with some  coefficients $\gamma_\nu$, but this is again unimproved compared to
the convolution itself. Either the coefficients $\gamma_\nu$ are connected with
Bernoulli numbers or with polynomials that have to be recursively computed.

\section{\texorpdfstring{$p$-Adic}{p-Adic} analysis}

Let $\ZZ_p$ be the ring of $p$-adic integers and $\QQ_p$ be the field of
$p$-adic numbers. Define $\ord_p x$ as the $p$-adic valuation of $x$.
Define $[x]$ as the integer part of $x \in \RR$.

\begin{lemma}[{\cite[p.~37]{Robert:2000}}] \label{lem:val-min}
If $n \geq  1$, then
\[
  \ord_p \sum_{\nu = 0}^n x_\nu \geq \min_{0 \leq \nu \leq n} \ord_p x_\nu
    \quad (x_\nu \in \QQ_p),
\]
where equality holds, if there exists an index $m$ such that
$\ord_p x_m < \ord_p x_\nu$ for all $\nu \neq m$.
\end{lemma}

\begin{lemma}[{\cite[p.~241]{Robert:2000}}] \label{lem:val-fac}
If $n \geq  1$, then
\[
   \ord_p n! = \frac{n - s_p(n)}{p-1},
\]
where $s_p(n)$ is the sum of the digits of the $p$-adic expansion of $n$.
\end{lemma}

For even $n > 0$ the numbers $\GNN_n$ are $p$-adically interesting for $p=2$,
whereas the numbers $\GN_n$ are odd integers.

\begin{prop} \label{prop:ord-gnn}
For $n \in 2\NN$ the numbers $\GNN_n \in \QQ_2 \backslash \ZZ_2$, while
$\GNN_n \in \ZZ_p$ for $p > 2$. More precisely, $\ord_2 \GNN_n = - \ord_2 n$
and $\ord_2 \GN_n = 0$.
\end{prop}

\begin{proof}
Let $n \in 2\NN$.
It is well known that the tangent numbers (cf.~\cite[p.~287]{Graham&others:1994})
\[
  2^n ( 1 - 2^n ) \BNN_n = 2^{n-1} \GNN_n
\]
are integers, here defined with different sign. The right-hand side follows by
\eqref{eq:gn-bn}. Hence the numbers $\GNN_n$ are
$p$-integers for $p > 2$. For $p=2$ we derive that
\[
  \ord_2 \GNN_n = \ord_2 ( 2 \BNN_n ) = - \ord_2 n,
\]
where we have used the fact that $\ord_2 ( 2 \BN_n) = 0$, which follows by the
von Staudt-Clausen theorem, see \cite[Theorem~3, p.~233]{Ireland&Rosen:1990}.
Since $n$ is even, we infer that $- \ord_2 n < 0$ and consequently that
$\GNN_n \in \QQ_2 \backslash \ZZ_2$. By the same arguments, it follows from
\eqref{eq:gn-bn} that $\ord_2 \GN_n = 0$.
\end{proof}

It remains of interest to evaluate the convolution $( \GNN_0 + \GNN_0 )^n$ for
even $n$. We will see that the $2$-adic valuation of $( \GNN_1 + \GNN_1 )^n$
has a simple form, while the $2$-adic valuation of $( \GNN_0 + \GNN_0 )^n$ is
more complicated.

\begin{prop}
Let $n \geq 1$ and $m = [n / 2]$. Then
\[
  \ord_2 \, ( \GNN_1 + \GNN_1 )^n = - \ord_2 (m+1).
\]
\end{prop}

\begin{proof}
By \eqref{eq:conv-gnn-11} it follows that
\[
  \ord_2 \, ( \GNN_1 + \GNN_1 )^n = 1 + \ord_2 ( \GNN_{n+2} + \GNN_{n+1} )
    =: 1 + g,
\]
where either $\GNN_{n+2}$ or $\GNN_{n+1}$ vanishes. Note that $n+2 = 2(m+1)$
for even $n$ and $n+1 = 2(m+1)$ for odd $n$. We conclude that $g = -\ord_2
(2(m+1))$ using Proposition~\ref{prop:ord-gnn}, which gives the result.
\end{proof}

\begin{prop} \label{prop:ord-conv-gnn-0}
We have
\[
  \ord_2 \, ( \GNN_0 + \GNN_0 )^n = \begin{cases}
      \infty, & \text{if $n = 1$}, \\
      1 - \ord_2 (n-1), & \text{if odd $n \geq 3$}, \\
      1 - \ord_2 n - [ \log_2 n ] + 2 \omega_2(n),
        & \text{if even $n \geq 2$},
    \end{cases}
\]
where $\omega_2(n) = 1$, if $n$ is a power of $2$, otherwise $\omega_2(n) = 0$.
\end{prop}

\begin{proof}
The cases $n=1, 2, 4$ are handled separately with $\GNN_1 = 1$ and
$\GNN_2 = -\tfrac{1}{2}$. For odd $n \geq 3$ we have $( \GNN_0 + \GNN_0 )^n =
2n \GNN_{n-1}$ by symmetry and different parity of indices. The result follows
by Proposition~\ref{prop:ord-gnn}. Now, let $n$ even and $n \geq 6$.
We first obtain by Lemma~\ref{lem:conv-rel} that
\begin{equation} \label{eq:loc-sum-gnn00}
  ( \GNN_0 + \GNN_0 )^n = \frac{2}{n} ( \GNN_0 + \GN_0 )^n.
\end{equation}
Set $\II = \{2,4,\ldots,n-2\}$ and $L = [ \log_2 n ] - \omega_2(n)$. Further
define $\ell_2(x)$ as the number of digits of $x \in \NN$ in base $2$.
Note that $\ord_2 \binom{n}{\nu} = - s_2(n) + s_2(\nu) + s_2(n-\nu)$ by
Lemma~\ref{lem:val-fac}. With the help of Proposition~\ref{prop:ord-gnn} and
Lemma~\ref{lem:val-min}, we can evaluate $\ord_2 \, ( \GNN_0 + \GN_0 )^n$ and
obtain that
\begin{equation} \label{eq:loc-ord-gnn00}
  \ord_2 \sum_{\nu \in \II} \binom{n}{\nu} \GNN_\nu \GN_{n-\nu}
    \geq - s_2(n) + \min_{\nu \in \II}
    \big( s_2(\nu) + s_2(n-\nu) - \ord_2 \nu \big).
\end{equation}
We will show that there is only one minimum on the right-hand side to get
equality. To be more precise, this takes place for $\nu_m = 2^L$, the greatest
power of $2$ in $\II$, where we have
\begin{equation} \label{eq:loc-ord-vm}
  s_2( \nu_m ) + s_2( n-\nu_m ) - \ord_2 \nu_m = 1 + s_2( n-2^L ) - L.
\end{equation}
We have now to distinguish between two cases, whether $n$ is a power of $2$ or
not.

Case $n=2^{L+1}$: The right-hand side of \eqref{eq:loc-ord-vm} reduces to
$2-L$. We derive for $\nu \in \II - \{ \nu_m \}$ that
\[
  2 - L < s_2(\nu) + s_2(n-\nu) - \ord_2 \nu,
\]
since $2 \leq s_2(\nu) + s_2(n-\nu)$ and $-L < - \ord_2 \nu$ by construction.

Case $n \neq 2^{L+1}$: One observes that $\ell_2(n) = \ell_2( \nu_m )$.
Regarding \eqref{eq:loc-ord-vm} we then conclude that $1 + s_2( n-2^L ) - L =
s_2(n) - L$. As above, for $\nu \in \II - \{ \nu_m \}$ we have
\[
  \ord_2 \nu - L < 0 \leq \ord_2 \binom{n}{\nu},
\]
which is equivalent to
\[
  s_2(n) - L < s_2(\nu) + s_2(n-\nu) - \ord_2 \nu.
\]
Both cases show that we have exactly one minimum. Thus
\eqref{eq:loc-ord-gnn00} becomes
\[
  \ord_2 \sum_{\nu \in \II} \binom{n}{\nu} \GNN_\nu \GN_{n-\nu}
    = - s_2(n) + s_2(\nu_m) + s_2(n-\nu_m) - \ord_2 \nu_m =: M,
\]
where we compute in case $n=2^{L+1}$ that
\[
  M = - s_2(n) + 2 - L = 1 - L = - [ \log_2 n ] + 2 \omega_2(n),
\]
otherwise $\omega_2(n) = 0$ and
\[
  M = - s_2(n) + s_2(n) - L = - L = - [ \log_2 n ] + 2 \omega_2(n).
\]
Together with \eqref{eq:loc-sum-gnn00} this gives the result.
\end{proof}

\section{Proof of Theorems}

\begin{proof}[Proof of Theorem \ref{thm:fs-value}]
\eqref{item:fs-1} We have two proofs either by Proposition~\ref{prop:drv-fs-bn}
or by Proposition~\ref{prop:int-fn}.
\eqref{item:fs-2} This is shown by Proposition~\ref{prop:val-fn}.
\eqref{item:fs-3} This is given by Corollary~\ref{corl:fs-sr}.
\end{proof}

\begin{proof}[Proof of Theorem \ref{thm:fh-value}]
\eqref{item:fh-1}  We have two proofs either by Proposition~\ref{prop:drv-fh-bn}
or by Proposition~\ref{prop:int-fn}.
\eqref{item:fh-2} This is given by Proposition~\ref{prop:val-fn}.
\eqref{item:fh-3} We have two different methods. The first proof is derived by
Proposition~\ref{prop:val-fn}. Second proof: For even $n$ we then obtain by
Corollary~\ref{corl:fh-sr} and Lemma~\ref{lem:semiring} that
\[
  \fh_n(x)/x - (n-1) \fs_{n-1}(x) = 0
\]
for $x=-\tfrac{1}{2}$.
\eqref{item:fh-4} For even $n$ we get by Propositions~\ref{prop:val-fn}
and \ref{prop:ord-gnn} that
\[
  \ord_2 \fh_n(-\tfrac{1}{2}) = -1 + \ord_2 \GNN_n = -1 - \ord_2 n.
\]
For odd $n$ we derive by Propositions~\ref{prop:val-fn},
\ref{prop:ord-conv-gnn-0}, and Eq.~\eqref{eq:conv-gnn-0} that
\begin{align*}
  \ord_2 \fh_n(-\tfrac{1}{2})
    &= \ord_2 \left( \frac{1}{2} ( \GNN_0 + \GNN_1 )^n \right)
    = \ord_2 \left( \frac{1}{4} ( \GNN_0 + \GNN_0 )^{n+1} \right) \\
    &= -1 - \ord_2( n+1 ) - [ \log_2( n+1 ) ] + 2 \omega_2(n+1).
\end{align*}
If $n+1 = 2^r$ with $r \geq 1$, then the latter expression simplifies to $-1 -
2(r-1)$, since $\ord_2( n+1 ) = [ \log_2( n+1 ) ] = r$ and $\omega_2(n+1) = 1$
in that case; otherwise $\omega_2(n+1)$ vanishes.
\eqref{item:fh-5} This is a consequence of Propositions \ref{prop:lambda-rec}
and \ref{prop:lambda-sr}.
\eqref{item:fh-6} This is Corollary~\ref{corl:fh-sr}.
\end{proof}

To prove Theorem \ref{thm:fsh-deriv} we have to introduce some transformations.
The Hadamard product (\cite[pp.~85--86]{Comtet:1974}) of two formal series
\begin{equation} \label{eq:poly-f-g}
  f(x) = \sum_{\nu \geq 0} a_\nu x^\nu, \quad
  g(x) = \sum_{\nu \geq 0} b_\nu x^\nu
\end{equation}
is defined to be
\[
  (f \odot g)(x) = \sum_{\nu \geq 0} a_\nu b_\nu x^\nu.
\]

For a sequence $(s_n)_{n \geq 0}$ its binomial transform $(s^*_n)_{n \geq 0}$
is defined by
\[
  s^*_n = \sum_{k=0}^{n} \binom{n}{k} (-1)^k s_k.
\]
Since the inverse transform is also given as above, we have
$(s^{**}_n)_{n \geq 0} = (s_n)_{n \geq 0}$, see \cite[p.~192]{Graham&others:1994}.
The following transformation is due to Euler (\cite[Ex.~3, p.~169]{Bromwich:1908}),
where we prove a finite case.

\begin{prop} \label{prop:hadamard-prod}
If $f, g$ are polynomials as defined in \eqref{eq:poly-f-g},
then the Hadamard product is given by
\begin{equation} \label{eq:hadamard-prod}
  (f \odot g)(x) = \sum_{\nu \geq 0} (-1)^\nu a^*_\nu \,
    \frac{g^{(\nu)}(x)}{\nu!} x^\nu.
\end{equation}
\end{prop}

\begin{proof}
We may assume that $f \cdot g \neq 0$. Let $N = \deg g$.
Define $g_n(x) = \sum_{\nu = 0}^n b_\nu x^\nu$ where $g_N(x) = g(x)$.
We use induction on $n$ up to $N$. For $n = 0$ we have
\[
  (f \odot g_0)(x) = a^*_0 \, g_0(x) = a_0 b_0.
\]
Now assume the result holds for $n \geq 0$.
Since $g_{n+1}(x)=b_{n+1} x^{n+1} + g_n(x)$, we consider the difference of
\eqref{eq:hadamard-prod} for $n+1$ and $n$. Thus
\begin{align*}
  (f \odot g_{n+1})(x) - (f \odot g_n)(x)
    &= \sum_{\nu = 0}^{n+1} (-1)^\nu a^*_\nu \, \frac{b_{n+1}(n+1)_\nu \,
       x^{n+1-\nu}}{\nu!} x^\nu \\
    &= b_{n+1} x^{n+1} \sum_{\nu = 0}^{n+1} \binom{n+1}{\nu} (-1)^\nu a^*_\nu \\
    &= a_{n+1} b_{n+1} x^{n+1}
\end{align*}
showing the claim for $n+1$.
\end{proof}

\begin{proof}[Proof of Theorem \ref{thm:fsh-deriv}]
The binomial transform
\begin{equation} \label{eq:bt-hn}
  -\frac{1}{n} = \sum_{k=1}^{n} \binom{n}{k} (-1)^k \HN_k
\end{equation}
is well known, cf.~\cite[pp.~281--282]{Graham&others:1994}.
Using Proposition~\ref{prop:hadamard-prod} with
$f(x) = \sum_{\nu = 1}^n \HN_\nu x^\nu$ and $g = \fs_n$,
where $f \odot g = \fh_n$, gives the result by means of \eqref{eq:bt-hn}.
\end{proof}

\appendix
\section{Figures}

\begin{minipage}{0.95\textwidth}
\captionsetup{type=figure}
\captionof{figure}{Functions $\fs_n$}
\label{fig:fs}
\begin{center}
  \includegraphics[width=9cm]{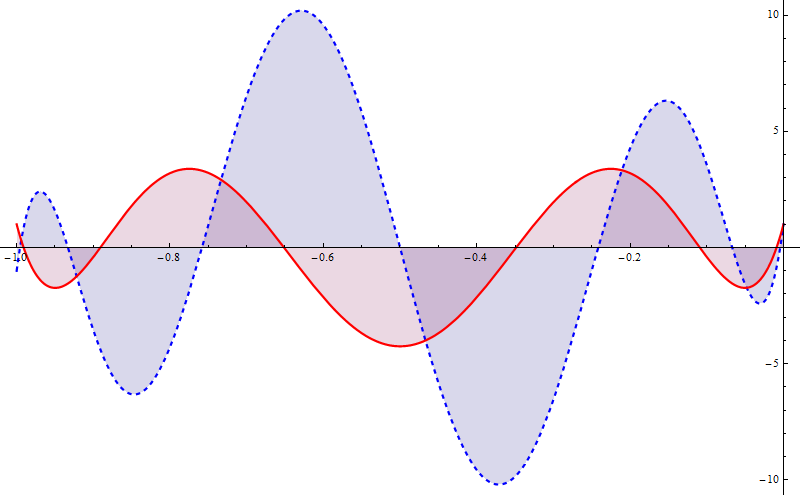}
\end{center}
\begin{center} \small
  $\fs_n(x)/x$ with $n=8$ (dashed blue line), 
  $\fs_n(x)/x$ with $n=7$ (red line).
\end{center}
\end{minipage}

\begin{minipage}{0.95\textwidth}
\captionsetup{type=figure}
\captionof{figure}{Functions $\fh_n$}
\label{fig:fh}
\begin{center}
  \includegraphics[width=9cm]{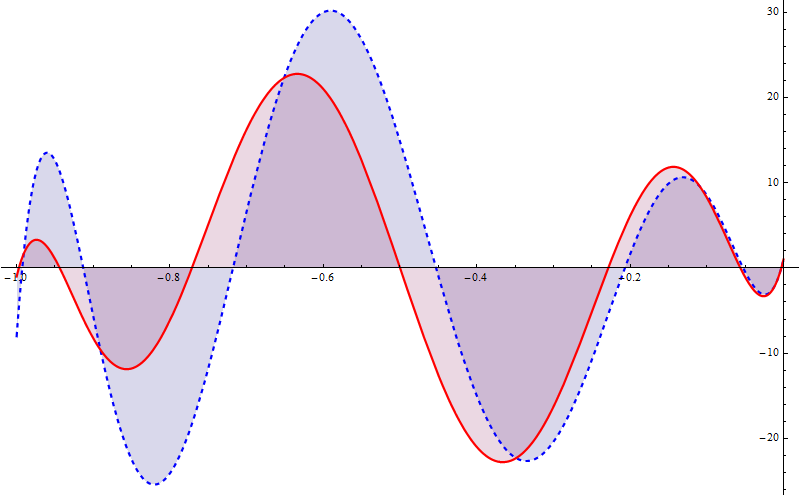}
\end{center}
\begin{center} \small
  $\fh_n(x)/x$ with $n=8$ (dashed blue line), \\
  $\fh_n(x)/x-(n-1)\fs_{n-1}(x)$ with $n=8$ (red line).
\end{center}
\end{minipage}

\section*{Acknowledgment}

We are grateful to the anonymous referee for valuable suggestions and motivating to
extend and improve the results; also for giving reference \cite{Sprugnoli:1994}.

\bibliographystyle{amsplain}

\bigskip

\end{document}